\documentclass{amsart}

\newtheorem{theorem}{Theorem}[section]
\newtheorem{lemma}[theorem]{Lemma}
\newtheorem{proposition}[theorem]{Proposition}
\newtheorem{corollary}[theorem]{Corollary}
\newtheorem{note}[theorem]{Note}
\theoremstyle{definition}
\newtheorem{definition}[theorem]{Definition}
\newtheorem{example}[theorem]{Example}

\theoremstyle{remark}
\newtheorem{remark}[theorem]{Remark}

\numberwithin{equation}{section}

\begin{document}

\title{Hausdorff dimension and infinitesimal similitudes on complete metric spaces}


%\address{Department of Mathematics, IIT Delhi, New Delhi, India 110016}

%\email{viswa@maths.iitd.ac.in}

\author{S. Verma}
\address{Department of Mathematics, IIT Delhi, New Delhi, India 110016 }
\email{saurabh331146@gmail.com}
%\thanks{The author was supported in part by UGC.}

%    General info

%\date{January 1, 2001 and, in revised form, June 22, 2001.}


\begin{abstract}

 In this paper, we answer a question of Nussbaum, Priyadarshi, and Lunel [Positive operators and Hausdorff dimension of invariant sets, Trans. Amer. Math. Soc. 364(2) (2012) 1029-1066.]. We also show that the Hausdorff dimension and box dimension of the attractor generated by a finite set of contractive infinitesimal similitudes are the same. Further, we extend many results of dimension theory to complete metric spaces. In the last part, we fill the gaps in the proofs of some articles, which are related to the dimension theory, and hint at some possible improvements in the recent papers.  
\end{abstract}

\maketitle

%%%%%%%%%%%%%%%%%%%%%%%%%%%%%%%%%%%%%%%%%%%%%%%%%%%%%%%%%%%%%%%%%%%%%%%%

%%%%%%%%%%%%%%%%%%%%%%%%%%%%%%%%%%%%%%%%%%%%%%%%%%%%%%%%%%%%%%%%%%%%%%%%
.

\section{INTRODUCTION}

Rooted in seminal work \cite{MW} of Mauldin and Williams on the graph-directed constructions, Nussbaum et al. \cite{Nussbaum1} proposed a generalized graph-directed systems, and studied the Hausdorff dimension of limit set for a finite family of contractive infinitesimal similitudes on a complete, perfect metric space. The concept of infinitesimal similitude introduced in \cite{Nussbaum1} generalizes not only the similitudes on general metric spaces but also the concept of conformal maps from Euclidean domain to general metric spaces. Therefore, the work in \cite{Nussbaum1} can be treated as a generalization of several works, see, for instance, \cite{Bandt,Lau,Lau2,MU1,MW,Peres,Schief2,Schief1}.

\par
Let $(X,d) $ be a compact and perfect metric space. Let $\mathcal{F}=\{X;f_1,f_2,\dots, f_N\}$ be an Iterated Function System (IFS) such that for $ 1 \le i \le N,$ $ f_i: X \rightarrow X$ is a contraction map with contraction coefficient $c_i.$ Then, by a result of Hutchinson \cite{JH}, there exists a unique, compact, non-empty set $A \subset X $, called attractor or limit set, with $$ A = \cup_{i=1}^{N}f_i(A).$$ More precisely, existence of the attractor $A$ is shown by a Hutchinson map $\mathcal{F}: \mathcal{H}(X) \to \mathcal{H}(X)$ defined by $\mathcal{F}(C)= \cup_{i=1}^{N} f_i(C),$ where $ \mathcal{H}(X) $ is a collection of nonempty compact subsets of $X$ equipped with Hausdorff metric induced by $d.$
Assume the map $f_i:X \rightarrow X$ is an infinitesimal similitude on $X$ and the map $x \to (Df_i)(x)$ is a strictly positive H\"older continuous function on $X$ for $1 \le i\le N.$ For $ \sigma \ge 0,$ define $L_{\sigma}: \mathcal{C}(X) \rightarrow \mathcal{C}(X)$ by $$ (L_{\sigma}g)(x)= \sum_{i=1}^{N}\big((Df_i)(x)\big)^{\sigma} g(f_i(x)).$$
By \cite[Theorem $5.4$]{Nussbaum4}, the operator $L_{\sigma}$ has a strictly positive eigenvector $u_{\sigma}$ with eigenvalue equal to the spectral radius $r(L_{\sigma})$ of $L_{\sigma}.$ 

Nussbaum, Priyadarshi and Lunel \cite{Nussbaum1} prove the following.
\begin{theorem}[\cite{Nussbaum1}, Theorem $1.2$]\label{thm-nussbaum}
Let $f_i: X \rightarrow X$ for $ 1 \le i \le N$ be infinitesimal similitudes and assume that the map $ x \to (Df_i)(x)$ is a strictly positive H\"older continuous function on $X.$ Assume that $f_i: X \rightarrow X $ is a contraction map with contraction coefficient $c_i$ and let $A$ denote the unique invariant set such that $$ A = \cup_{i=1}^{N}f_i(A).$$ Further, assume that $f_i,~ 1\le i \le N,$ satisfy $$ f_i(A) \cap f_j(A) = \emptyset~ \text{for} ~1 \le i,j \le N, i \ne j$$
and are one-to-one on $A.$ Then the Hausdorff dimension of $A$ is given by the unique $\sigma_0$ such that $r(L_{\sigma_0})=1.$
\end{theorem}
Note that the strong separation condition(SSC), that is, $$ f_i(A) \cap f_j(A) = \emptyset~ \text{for} ~1 \le i,j \le N, i \ne j,$$ is very strong, because it will cover only cantor type sets which have less importance over connected sets.
The authors of \cite{Nussbaum1} posed a question that whether the above result holds if we assume the strong open set condition(SOSC) instead of SSC. The present article gives an affirmative answer to the question. Though the OSC and SOSC are equivalent in Euclidean spaces for the IFS consisting of similitudes \cite{Bandt,Schief1} and conformal maps \cite{Lau,Lau2,Peres,Ye}. But we should emphasize that the open set condition(OSC) and the SOSC are not equivalent for the IFS consisting even similitudes in complete metric spaces, see, \cite{Schief2}.

Our paper is a continuation of work reported in \cite{Nussbaum1}. 
\subsection{Hausdorff dimension and Box dimension }
 
Let $(X,d)$ be a separable metric space. If $U$ is any non-empty subset of $X,$ the diameter of $U$ is defined as $$|U|=\sup\{d(x,y): x,y \in U\}.$$ Suppose $F$ is a subset of $X$ and $s$ is a non-negative real number. The $s-$dimensional Hausdorff measure of $F$ is defined as $$H^s(F)= \lim_{\delta \rightarrow 0^+} \Big[\inf \Big\{\sum_{i=1}^{\infty}|U_i|^s:  F \subseteq \cup_{i=1}^{\infty} U_i ~\text{and}~ |U_i| < \delta \Big\} \Big] .$$
\begin{definition} (\cite{Mattila,Fal1}) 
Let $F \subset X$ and $s \ge 0.$ The Hausdorff dimension of $F$ is $$ \dim_H(F)=\inf\{s:H^s(F)=0\}=\sup\{s:H^s(F)=\infty\}.$$
\end{definition}

\begin{definition} (\cite{Fal1}) 
Let $F$ be any non-empty bounded subset of $X$ and let $N_{\delta}(F)$ be the smallest number of sets of diameter at most $\delta$ which can cover $F.$ The lower box dimension and upper box dimension of $F$ respectively are defined as $$\underline{\dim}_B(F)=\varliminf_{\delta \rightarrow 0^+} \frac{\log N_{\delta}(F)}{- \log \delta}$$
and
$$\overline{\dim}_B(F)=\varlimsup_{\delta \rightarrow 0^+} \frac{\log N_{\delta}(F)}{- \log \delta}.$$
If above two are equal, we call the common value as box dimension of $F,$   $$\dim_B(F)=\lim_{\delta \rightarrow 0^+} \frac{\log N_{\delta}(F)}{- \log \delta}.$$
\end{definition}
For basic properties related to the above concepts, we refer the reader to \cite{Fal1,Mattila}.
\begin{definition}
The Hausdorff dimension of a measure $\mu$ is defined to be $$ \dim_H(\mu)=\inf\{\dim_H(A): \mu(X \backslash A)=0\}.$$
\end{definition}

\section{Infinitesimal-similitude}

Let $(X,d_X)$ be a compact, perfect metric space and $(Y,d_Y)$ be a metric space. Let $f: X \rightarrow Y $ be a function. We define a set-valued map $(Df)^* :X \to \mathbb{R}$ as follows
\begin{equation*}
\begin{aligned}
(Df)^*(x)=\Bigg\{ & \lim_{n \to \infty} \frac{d_Y(f(x_n),f(y_n))}{d_X(x_n,y_n)}:  \text{for some}~ (x_n), (y_n) ~\text{with}~ x_n \ne y_n ~ \text{for}~ n \ge 1\\ & \text{and} ~ x_n \to x , y_n \to x \Bigg\}.
\end{aligned}
\end{equation*}
\begin{definition}
A mapping $f:X \to Y$ is said to an infinitesimal similitude at $x$ if $(Df)^*(x)$ is a nonempty and singleton set. We denote $(Df)^*$ by simply $(Df)$. Further, if $f$ is an infinitesimal similitude at $x$ for all $x \in X$ then we say that $f$ is an infinitesimal similitude on $X$. 
\end{definition}

\par
\begin{theorem}
The set $(Df)^*(x)$ is a closed subset of $\mathbb{R}$ for each $x \in X.$
\end{theorem}
\begin{proof}
Let $x \in X.$ If $(Df)^*(x) = \emptyset$, then nothing to prove. 
Let $(z_n)$ be a sequence in $(Df)^*(x)$ and $ z_n \to z$ as $ n \to \infty.$
To show $ z \in (Df)^*(x),$ we proceed as follows.
By definition of $(Df)^*(x)$, choose $x_{n,m} \ne y_{n,m}$ with $x_{n,m} \to x, y_{n,m} \to x $ as $m \to \infty$ and $$\lim_{m \to \infty} \frac{d_Y(f(x_{n,m}),f(y_{n,m}))}{d_X(x_{n,m},y_{n,m})}= z_n.$$ 
For sufficiently large $N_n \in \mathbb{N}$,
$$ \Bigg| \frac{d_Y(f(x_{n,N_n}),f(y_{n,N_n}))}{d_X(x_{n,N_n},y_{n,N_n})} - z_n \Bigg| < \frac{1}{n}.$$
This with triangle inequality produces
$$ \Bigg| \frac{d_Y(f(x_{n,N_n}),f(y_{n,N_n}))}{d_X(x_{n,N_n},y_{n,N_n})} - z \Bigg| \le
\Bigg| \frac{d_Y(f(x_{n,N_n}),f(y_{n,N_n}))}{d_X(x_{n,N_n},y_{n,N_n})} - z_n \Bigg|+ |z_n -z|,$$ hence the claim.
\end{proof}
\begin{example}
Let $f: \mathbb{C} \to \mathbb{C}$ be a map defined by $f(z)= \bar{z}$. Observe that $|f(z)-f(w)| =|z-w|~ \forall z,w \in \mathbb{C}$, and $f$ is nowhere $\mathbb{C}-$differentiable. Also note that $(Df)(z)=1$ for every $z \in \mathbb{C}$, that is, $f$ is an infinitesimal similitude.
\end{example}
\begin{example}
Let $f: \mathbb{R} \to \mathbb{R}$ be a map defined by $f(x)= |x|$. Observe that $|f(x)-f(y)| \le |x-y|~ \forall x,y \in \mathbb{R}$, and $f$ is not differentiable at $x=0.$ By taking sequences $x_n= \frac{1}{n}, y_n = \frac{-1}{n}$ we get $ \lim_{n \to \infty} \frac{\big||x_n| -|y_n|\big|}{|x_n -y_n|}=0.$ For sequences $x_n= \frac{1}{n}, y_n = 0$, we get $ \lim_{n \to \infty} \frac{\big||x_n| -|y_n|\big|}{|x_n -y_n|}=1.$
Hence $f$ is not an infinitesimal similitude at $x=0.$ 
\end{example}
\begin{example}
Define $f: \mathbb{R} \to \mathbb{R}$ by \begin{equation*} f(x) =
       \begin{cases}
     x^2 \sin(\frac{1}{x}),~~ \text{if}~~ x \ne 0 \\
          0 ,~~ \text{otherwise}.
      \end{cases}
   \end{equation*}
   Then $f$ is differentiable on $\mathbb{R}.$ We also have
   \begin{equation*} f'(x) =
          \begin{cases}
        2x \sin(\frac{1}{x})- \cos(\frac{1}{x}),~~ \text{if}~~ x \ne 0 \\
             0 ,~~ \text{otherwise}.
         \end{cases}
      \end{equation*}
    Note that $f'$ is not continuous at $x=0.$ Also $f$ is not an infinitesimal similitude at $x=0.$
\end{example}
\begin{theorem}
Let $f:\mathbb{R} \to \mathbb{R}$ be a differentiable function. Then the $Df$ exists at $x_0$ if and only if modulus of the derivative $|f'|:\mathbb{R} \to \mathbb{R}$ is continuous at $x_0$. In particular, if $Df$ exists then $Df=|f'|.$
\end{theorem}
\begin{proof}
Suppose $Df$ exists at $x_0$. Then, for $x_n \to x_0$, $x_n \ne x$, $$ Df(x_0)= \lim_{n \to \infty} \frac{|f(x_n)-f(x_0)|}{|x_n -x_0|} =\bigg|\lim_{n \to \infty} \frac{f(x_n)-f(x_0)}{x_n -x_0} \bigg|=|f'(x_0)|.$$
By Lemma \ref{conti}, $|f'| $ is continuous at $x_0.$ Now, suppose $|f'|:\mathbb{R} \to \mathbb{R}$ is continuous at $x_0$. Let $x_n \ne y_n$ such that $x_n \to x_0$ and $y_n \to x_0$. By mean value theorem, $$\bigg|\frac{f(x_n)-f(y_n)}{x_n -y_n}\bigg|= |f'(\xi_n)|,$$ where either $\xi_n \in (x_n,y_n)$ or $\xi_n \in (y_n,x_n).$ Since $x_n \to x_0$ and $y_n \to x_0$, we get $\xi_n \to x_0$. Continuity of $|f'|$ at $x_0$ in turn yields that $Df$ exists at $x_0.$

\end{proof}
\begin{theorem}
Let $f,g:X \to \mathbb{R}$ be infinitesimal similitudes at $x_0 \in X$. Let $K \in \mathbb{R}$. Then we have the following
\begin{enumerate}
\item The function $Kf$ is an infinitesimal similitude at $x_0$, and $\big(D(Kf)\big)(x_0) =|K|(Df)(x_0).$
\item If $\big(D(f+g)\big)(x_0)$ exists then $ \big(D(f+g)\big)(x_0) \le (Df)(x_0) + (Dg)(x_0).$
\item If $\big(D(fg)\big)(x_0)$ exists then $ \big(D(fg)\big)(x_0) \le |g(x_0)|~(Df)(x_0) + |f(x_0)|~(Dg)(x_0).$
\item If $g(x_0) \ne 0$ and $\bigg(D\big(\frac{f}{g}\big)\bigg)(x_0)$ exists then $$\bigg(D\big(\frac{f}{g}\big)\bigg)(x_0) \le \frac{|g(x_0)|~(Df)(x_0)-|f(x_0)|~(Dg)(x_0)}{g^2(x_0)}.$$
\end{enumerate}
\end{theorem}
\begin{remark}
We emphasize on that fact that in the above theorem strict inequality can occur. For example, define $f(x)=x$ and $g(x)= -x$. Here $(Df)(x)=1$ and $(Dg)(x)=1$ for each $x \in \mathbb{R}$, but $f+g =0$, which implies that $$0 = (D(f+g))(x)< (Df)(x)+ (Dg)(x)=2$$ for each $x \in \mathbb{R}$.
\end{remark}

\begin{lemma}[\cite{Nussbaum1}, Lemma $1.1$]
The map $\sigma \to r(L_{\sigma})$ is continuous and strictly decreasing. Furthermore, there is a unique $\sigma_0 \ge 0$ such that $r(L_{\sigma_0})=1.$
\end{lemma}

\begin{lemma}[\cite{Nussbaum1}, Lemma $4.1$]\label{conti}
If $f: X \rightarrow Y$ is an infinitesimal similitude, then $x \to (Df)(x)$ is continuous.
\end{lemma}
\begin{lemma}[\cite{Nussbaum1}, Lemma $4.2$] \label{Chain_lemma}
Let $f: X \rightarrow Y$ and $h: Y \rightarrow Z$ be given. If $f$ is an infinitesimal similitude at $x \in X$ and $h$ is an infinitesimal similitude at $f(x) \in Y$, then $hof$ is an infinitesimal similitude at $x \in X$ and $$ (D(hof))(x)= (Dh)(f(x))(Df)(x).$$
\end{lemma}
Our main theorem is the following.
\begin{theorem}\label{mainthm}
Let $f_i: X \rightarrow X$ for $ 1 \le i \le N$ be infinitesimal similitudes and assume that the map $ x \to (Df_i)(x)$ is a strictly positive H\"older continuous function on $X.$ Assume that $f_i: X \rightarrow X $ is a contraction map with contraction coefficient $c_i$ and let $A$ denote the unique invariant set such that $$ A = \cup_{i=1}^{N}f_i(A).$$ Further, assume that $f_i,~ 1\le i \le N,$ are one-to-one on $A$ and satisfy strong open set condition. Then the Hausdorff dimension of $A$ is given by the unique $\sigma_0$ such that $r(L_{\sigma_0})=1.$
\end{theorem}
\begin{proof}
Let $U$ be an open set originated from strong open set condition for $A$. Since $U \cap A \ne \emptyset $, we have an index $k \in I^*$ with $A_k \subset U$, where $I^*:=\cup_{m \in \mathbb{N}}\{1,2,\dots, N\}^m$, that is, the set of all finite sequences made up of the elements of $I:=\{1,2,\dots , N\}$ and $A_k:=f_k(A):=f_{k_1} \circ f_{k_2}\circ \dots \circ f_{k_m} (A)$ for $k \in I^m$($m$-times Cartesian product of $I$ with itself) and for $m \in \mathbb{N}.$ Now, we could see that for any (but fixed) $n$ and $j \in I^n$, the sets $A_{jk}$ are disjoint. Furthermore, the IFS $\{f_{jk}: j \in I^n\}$ satisfies the assumptions of Theorem \ref{thm-nussbaum}. Therefore, we have $   \dim_H(A^*)=\sigma_n$, where $A^* $ is an attractor of the aforesaid IFS and $\sigma_n$ is the unique number such that $r(L_{\sigma_n})=1$ and $L_{\sigma_n}: \mathcal{C}(X) \rightarrow \mathcal{C}(X)$ is defined by $$(L_{\sigma_n}g)(y)= \sum_{j \in I^n}\big((Df_{jk})(y)\big)^{\sigma_n} g(f_{jk}(y)).$$
 Since $A^* \subset A$, we have $$ \sigma_n \le \dim_H(A) \le \sigma_0.$$ 
With the help of Lemma \ref{Chain_lemma}, we get $$ (Df_{jk})(y)=(Df_j)(f_k(y))(Df_k)(y).$$
From the above, we can write
\begin{equation}\label{Eqn1}
\begin{aligned}
 (L_{\sigma_n}g)= \sum_{j \in I^n}\big((Df_{j})(f_k(y))\big)^{\sigma_n} \big((Df_{k})(y)\big)^{\sigma_n} g(f_j(f_k(y)).
 \end{aligned}
 \end{equation}
Let $h$ be a strictly positive eigenvector corresponding to eigenvalue $r(L_{\sigma_n}).$ That is, $L_{\sigma_n}h = r(L_{\sigma_n}) h.$ Since $r(L_{\sigma_n})=1,$ we get $L_{\sigma_n}h =  h.$ From the Equation \ref{Eqn1}, we obtain
\begin{equation}
\begin{aligned}
h(y)&=\sum_{j \in I^n}\big((Df_{j})(f_k(y))\big)^{\sigma_n} \big((Df_{k})(y)\big)^{\sigma_n} h(f_j(f_k(y))\\&= \big((Df_{k})(y)\big)^{\sigma_n}\sum_{j \in I^n}\big((Df_{j})(f_k(y))\big)^{\sigma_n}  h(f_j(f_k(y))
\end{aligned}
\end{equation}
Suppose for contradiction $ \beta := \dim_H(A)< \sigma_0.$ Since $0 < (Df_i)(z) \le c_i <1, \forall ~i \in \{1,2,\dots ,N\},$ with $m_h=\min\{h(x)\}$ and $c_{\max}=\{c_1,c_2,\dots, c_N\}$ we immediately obtain
\begin{equation}\label{Eqn2}
\begin{aligned}
h(y)\big((Df_{k})(y)\big)^{-\sigma_n}& \ge  \sum_{j \in I^n}\big((Df_{j})(f_k(y))\big)^{\beta}  h(f_j(f_k(y))\\&=
\sum_{j \in I^n}\big((Df_{j})(f_k(y))\big)^{\sigma_0} \big((Df_{j})(f_k(y))\big)^{\beta-\sigma_0} h(f_j(f_k(y))\\&\ge
m_h ~c_{max}^{n(\beta- \sigma_0)}\sum_{j \in I^n}\big((Df_{j})(f_k(y))\big)^{\sigma_0}
\end{aligned}
\end{equation}
Let $h_0$ be a strictly positive eigenvector corresponding to eigenvalue $r(L_{\sigma_0}).$ That is, $L_{\sigma_0}h_0 = r(L_{\sigma_0}) h.$ Since $r(L_{\sigma_0})=1,$ we get $L_{\sigma_0}h_0 =  h_0.$ By definition of the operator $L_{\sigma_0},$ we have
\begin{equation}\label{Eqn2main}
\begin{aligned}
h_0(y)&=\sum_{i=1}^{N}\big((Df_{i})(y)\big)^{\sigma_0} h_0(f_i(y)).
\end{aligned}
\end{equation}
Now we estimate
\begin{equation}\label{Eqn22}
\begin{aligned}
\sum_{j \in I^n}\big((Df_{j})(f_k(y))\big)^{\sigma_0} =& \sum_{j \in I^n}\big((Df_{j_1 j_2 \dots j_n})(f_k(y))\big)^{\sigma_0}\\= & \sum_{j \in I^n}\big((Df_{j_1})(f_{j_2}\dots f_{j_n}f_k(y))\big)^{\sigma_0} \big((Df_{j_2})(f_{j_3}\dots f_{j_n}f_k(y))\big)^{\sigma_0} \dots \big((Df_{j_n})(f_k(y))\big)^{\sigma_0} \\= &\sum_{j_n=1}^{N}\big((Df_{j_n})(f_k(y))\big)^{\sigma_0}\sum_{j_{n-1}=1}^{N}\big((Df_{j_{n-1}})(f_{j_n}f_k(y))\big)^{\sigma_0}\\ & \dots  \sum_{j_1=1}^{N}\big((Df_{j_1})(f_{j_2}\dots f_{j_n}f_k(y))\big)^{\sigma_0}\\ \ge &\sum_{j_n=1}^{N}\big((Df_{j_n})(f_k(y))\big)^{\sigma_0}\sum_{j_{n-1}=1}^{N}\big((Df_{j_{n-1}})(f_{j_n}f_k(y))\big)^{\sigma_0}\\& \dots \sum_{j_1=1}^{N}\big((Df_{j_1})(f_{j_2}\dots f_{j_n}f_k(y))\big)^{\sigma_0} \frac{h_0(f_{j_1}f_{j_2}\dots f_{j_n}f_k(y))}{\max{h_0}}
\\= & \frac{1}{\max{h_0}}\sum_{j_n=1}^{N}\big((Df_{j_n})(f_k(y))\big)^{\sigma_0}\sum_{j_{n-1}=1}^{N}\big((Df_{j_{n-1}})(f_{j_n}f_k(y))\big)^{\sigma_0}\\& \dots \sum_{j_2=1}^{N}\big((Df_{j_2})(f_{j_3}\dots f_{j_n}f_k(y))\big)^{\sigma_0} h_0(f_{j_2}f_{j_3}\dots f_{j_n}f_k(y))\\
= & \frac{h_0(f_k(y))}{\max{h_0}},
\end{aligned}
\end{equation}
the second equality follows from Lemma \ref{Chain_lemma} and the last two equalities follow from Equation \ref{Eqn2main}. 
Using Equations \ref{Eqn2} and \ref{Eqn22}, we establish the following
\begin{equation}
\begin{aligned}
h(y)\big((Df_{k})(y)\big)^{-\sigma_n}&\ge
\frac{m_h h_0(f_k(y))}{\max{h_0}}c_{max}^{n(\beta- \sigma_0)}\\ &\ge 
\frac{m_h \min{h_0}}{\max{h_0}}c_{max}^{n(\beta- \sigma_0)}.
\end{aligned}
\end{equation}
Since $ c_{max} < 1$ and the term on left side in the above expression is bounded, we have a contradiction as $n $ tends to infinity. Thus our supposition were wrong. This implies that $ \dim_H(A) \ge  \sigma_0,$ which is the required result.
\end{proof}
\begin{remark}
Above theorem serves as an addendum to the paper of Nussbaum \cite{Nussbaum1}. It also improves some results in \cite{AP,Nussbaum2,Nussbaum3}, see the last section.
\end{remark}
\begin{remark}
For $h_0(x)=\sum_{i=1}^{N}\big((Df_{i})(x)\big)^{\sigma_0} h_0(f_i(x))$ we have $$\frac{\min{h_0}}{\max{h_0}} \le \frac{h_0(x)}{\max{h_0}} \le  \sum_{i=1}^{N}\big((Df_{i})(x)\big)^{\sigma_0} \le \frac{h_0(x)}{\min{h_0}} \le \frac{\max{h_0}}{\min{h_0}}.$$
\end{remark}
\begin{remark}
Here we note an interesting relation between Hausdorff dimensions of attractor $A$, eigenvector $u$ and $\mu.$
First we see that $\mathcal{F}(A)=A$, $L_{\sigma_0}u=u$
and $(L_{\sigma_0})^*(\mu) =\mu.$ It can be straightforwardly obtained that $$ \dim_H(\mu) \le \dim_H(A) \le \dim_H(G_u),$$ where $G_u$ denotes the graph of function $u.$
\end{remark}
\par

With the notation $A_i= f_i(A)$, $r_i= \inf_{x \in X} (Df_i)(x)$ and $R_i= \sup_{x \in X} (Df_i)(x)$ we have the following lemma
\begin{lemma}\label{new1}
Let $X$ and $f_i$ be as in Theorem \ref{mainthm}. Then 
there exists $c_1 >1 $ such that 
$R_i \le c_1 r_i $ for any $i \in I^* $ and $ \frac{r_ir_j}{c_1} \le r_{ij} \le c_1 r_i r_j$ for every $i,j \in I^*.$
\end{lemma}
\begin{proof}
We first note that the map $ x \mapsto (Df_i)(x) $ is H\"older continuous with exponent $s>0,$ that is, $$ |(Df_i)(x) -(Df_i)(y)| \le K d(x,y)^s ~~\forall ~~x,y \in X,$$ for some constant $K>0.$
Let $x,y \in X.$ Then, by mean value theorem, there exists $\xi$ between $(Df_i)(x)$ and $(Df_i)(y)$ such that
\begin{equation}
\begin{aligned}
|\ln((Df_i)(x))- \ln((Df_i)(y))|& = \frac{1}{\xi} | (Df_i)(x) -(Df_i)(y)|\\ & \le \frac{K}{\xi}  d(x,y)^s\\ & \le \frac{K}{m} d(x,y)^s,
\end{aligned}
\end{equation}
the last inequality follows because $Df_i$ is strictly positive, that is, $0<m < (Df_i)(x) < 1, ~~\forall~ x \in X.$ 
Therefore, we have $$ (Df_i)(x) \le (Df_i)(y) \exp\Big(\frac{K}{m} d(x,y)^s\Big).$$
Without loss of generality we assume that $\text{diam}(A) \le 1.$ Now for a suitable constant $c_1 > 1$, we get $$ (Df_i)(x) \le c_1 (Df_i)(y)~ \forall x,y \in A.$$ On taking infimum over all $y \in A$, we have 
\begin{equation} 
\begin{aligned}
  (Df_i)(x) & \le c_1 \inf_{y \in A}(Df_i)(y)~ \forall x \in A \\ & = c_1 r_i ~ \forall x \in A.
\end{aligned}
\end{equation}  
Consequently,  $ R_i \le c_1 r_i.$

For other one, we proceed as follows
\begin{equation} 
\begin{aligned}
(Df_{ij})(x) & = (Df_i)(f_j(x)) (Df_j)(x) \\ & \le (Df_j)(x) \sup_{x \in X}(Df_i)(f_j(x)) \\ & \le R_i (Df_j)(x)
\end{aligned}
\end{equation}
Now, we apply infimum on both sides, and obtain 
\begin{equation} 
\begin{aligned}
 \inf_{x \in X}(Df_{ij})(x) & \le R_i \inf_{x \in X}(Df_j)(x)\\
  r_{ij} & \le R_i r_j  \\ & \le c_1 r_i r_j.
\end{aligned}
\end{equation}
The last inequality follows from the first part of the lemma.
To get other side of inequality, we proceed in the following way.
\begin{equation} 
\begin{aligned}
& \frac{(Df_{ij})(x)}{(Df_i)(f_j(x))} = (Df_j)(x) ~  \text{since}~(Df_j)(z)> 0 ~ \forall z \in X \\ &
\frac{(Df_{ij})(x)}{\inf_{x \in X}(Df_i)(f_j(x))} \ge (Df_j)(x)
\\ &
\frac{(Df_{ij})(x)}{r_i} \ge (Df_j)(x)
\\ &
\frac{R_{ij}}{r_i} \ge (Df_j)(x)
\end{aligned}
\end{equation}
From the first part, one gets $$ \frac{c_1 r_{ij}}{r_i} \ge (Df_j)(x) \ge r_j.$$ Finally, we have $ r_{ij} \ge \frac{r_i r_j}{c_1}.$
\end{proof}
Nussbaum et al. \cite{Nussbaum1} proved a result analogous to ``mean value theorem" in terms of infinitesimal similitudes. In particular, Lemma $4.3$ of \cite{Nussbaum1} assumed a condition that $\theta$ is Lipschitz which has not been used in the proof. Therefore, the next lemma can be seen as a modification of \cite[Lemma $4.3$]{Nussbaum1}. 
\begin{lemma}\label{new2}
Assume that $(Df_j)(x)> 0 $ for every $x \in X.$ Then there exists $c_2 \ge  c_1$ and $\delta > 0$ such that for $x,y  \in X$, $d(x,y) < \delta,$ $$ \frac{(Df_j)(x)}{c_2} \le \frac{d(f_j(x),f_j(y))}{d(x,y)} \le c_2 (Df_j)(x)$$ for every $j \in I^*.$
\end{lemma}
\begin{proof}
%Lemma \ref{conti} implies that $ x \mapsto (Df_j)(x)$ is continuous on $X.$ 
We start by defining a function $F:X \times X \to \mathbb{R}$ by 
\begin{equation}
F(x,y)= 
\begin{cases}
\frac{d(f_j(x),f_j(y))}{d(x,y)}, ~~ \text{if}~~ x \ne y \\
(Df_j)(x),~~~~~\text{if} ~~ x=y.
\end{cases}
\end{equation} 
In view of Lemma \ref{conti}, one could see that function $F$ is continuous on $X \times X.$ Further, since $(Df_j)(x) >0 ,$ define $$G(x,y):= \frac{F(x,y)}{(Df_j)(x)}.$$
Note that $G(x,x)=1$ and $G$ is continuous. Using the compactness of $X \times X,$ $G$ is uniformly continuous on $X \times X.$
Equivalently, for a given $c >1$ there exists $\epsilon > 0$ (depending on $c$) with $$ c^{-1} < G(x,y) < c $$ for every $x,y \in X \times X$ with $d(x,y) < \epsilon.$
For $c_2 \ge c_1$, this immediately gives $$ c_2^{-1} <\frac{F(x,y)}{(Df_j)(x)} < c_2.$$ Hence the proof is complete.
\end{proof}
\begin{remark}
The assumption in the above lemma can not be relaxed. Note that any H\"older continuous function $f:X \to X $ with H\"older exponent greater than $1$, by definition of infinitesimal similitude, we get $ (Df)(x)=0~\forall x \in X.$
\end{remark}
\begin{theorem}
Let $\{X; f_1, f_2,\dots,f_N\}$ be an IFS consisting of contracting infinitesimal similitudes such that $(Df_i)(x) > 0 $ for every $x \in X$ and $i=1,2, \dots,N.$ Let $A$ be the associated attractor of the IFS, then $$\dim_H(A)= \underline{\dim}_B(A)=\overline{\dim}_B(A).$$
\end{theorem}
\begin{proof}
By Lemmas \ref{new1} and \ref{new2}, there is a constant $c>1$ and $\delta> 0$ such that
\begin{equation}\label{HBE}
\begin{aligned}
c^{-1} r_i \le \frac{d(f_i(x),f_i(y))}{d(x,y)} \le c r_i
\end{aligned}
\end{equation}
holds for every $x,y \in A$ with $d(x,y) < \delta$, and $i \in I^*.$
From the very definition of attractor, for each $z \in A$ there exists $i \in I^{\infty}$ such that $\lim_{n \to \infty} f_{i_1}\circ f_{i_2}\circ \dots \circ f_{i_n}(A) = z,$ where limit is taken with respect to the Hausdorff metric induced by the metric $d$. Therefore, for each $0< r< \frac{\delta}{c},$ we choose a least natural number $n=n(r)$ such that $$\min_{1 \le i\le N}r_i \frac{\delta}{c} \le c r r_{i|_n} < \delta  .$$
In view of Equation \ref{HBE}, we have $f_i(x),f_i(y) \in B_{\delta}(z)$ whenever $x,y \in B_r(z).$ Hence, we get $$ \frac{\min_{1 \le i\le N}r_i \delta}{c^2 r}d(x,y) \le d(f_i(x),f_i(y)) \le \frac{\delta}{c} d(x,y).$$
On defining a mapping $\Phi=f_{i_n}|_{B_r(z) \cap A}$, \cite[Theorem $3$]{Fal3} completes the proof.
\end{proof}

\begin{lemma}\label{new3}
There exists $c_3 \ge c_2 $ such that for any $x, y \in X$ we have $$ d(f_j(x),f_j(y)) \le c_3 r_j d(x,y)~~ j \in I^*.$$
\end{lemma}
\begin{proof}
Since $X$ is compact and connected, with the same $\delta $ as in Lemma \ref{new2} we have $$X \subseteq \cup_{k=1}^n B(x_k,\delta).$$
Now let $x,y \in X$ then using Lemmas \ref{new1} and \ref{new2} we get 
\begin{equation}
\begin{aligned}
d(f_j(x),f_j(y)) & \le d(f_j(x),f_j(x_1))+d(f_j(x_1),f_j(y_1))+\dots + d(f_j(y_n),f_j(y)) \\  & \le c_2 R_j (d(x,x_1)+d(x_1,y_1)+ \dots +d(y_n,y))\\ & \le 2n c_2 R_j d(x,y) \\ & \le 2nc_1 c_2 r_j d(x,y).
\end{aligned}
\end{equation}
With $c_3= 2nc_1c_2$ we have the required result.
\end{proof}

We may assume without loss of generality that $ \overline{X^o}=X$ and $B(A, \delta) \subset X^o.$ We let $ 0 < \epsilon < 2^{-1}c_3^{-1} \delta$, then $2 c_3 \epsilon \le \delta.$ We have
\begin{equation}\label{eqn2.7}
\begin{aligned} 
 B(A,c_3 \epsilon) \subset X.
 \end{aligned}
 \end{equation}
For $ \text{j} \in I^*$, we write $O_j:= f_j (B(A,\epsilon)).$
\begin{remark}
By Equation \ref{eqn2.7} and Lemma \ref{new2} we have for every $x \in A$ 
\begin{equation}\label{eqn2.8}
\begin{aligned}
B(f_j(x),c_2^{-1}\epsilon r_j) \subseteq f_j(B(x,\epsilon)) \subseteq B(f_j(x),c_2\epsilon r_j) .
\end{aligned}
\end{equation}
Hence we have 
\begin{equation}\label{eqn2.9}
\begin{aligned}
B(A_j,c_2^{-1}\epsilon r_j) & = \cup_{x \in A}B(f_j(x),c_2^{-1}\epsilon r_j) \\ & \subseteq \cup_{x \in A} f_j(B(x,\epsilon)) \\ & =f_j(B(A,\epsilon)) \\ & =O_j \\ & \subseteq \cup_{x \in A} B(f_j(x),c_2\epsilon r_j)\\ & = B(A_j,c_2\epsilon r_j ) .
\end{aligned}
\end{equation}
\end{remark}
Motivated by \cite{Ye}, we introduce some terminologies which will be used to develop more results of dimension theory in complete metric spaces.  
For $ 0 <b < 1$ we define $$ \Gamma_b=\{j=j_1j_2\dots j_n: r_{j_1j_2\dots j_n}< b \le r_{j_1j_2\dots j_{n-1}} \}.$$
Following an inductive way, see, for instance, \cite{Ye}, of defining the index set $\Gamma(j), j \in I^*:$
For $j \in I=\{1,2,\dots , N\}$, we define $$ \Gamma(j)=\{i \in \Gamma_{| O_j|}: A_i \cap O_j \ne  \emptyset \}.$$
Suppose $\Gamma(j)$ for $j \in I^*$ is defined, we define for $1 \le k \le N$, $$ \Gamma(kj)= \mathcal{B} \cup \mathcal{C} $$
where $$\mathcal{B}=\{ki:i \in \Gamma(j)\}$$ and $$\mathcal{C}=\{i \in \Gamma_{|O_{kj}|}: i_1 \ne k ~\text{and}~~ A_i \cap O_{kj}\}.$$
One could observe that each $i \in \Gamma(j)$ is of either type $\mathcal{B}$ or $\mathcal{C}$, $A_i \cap O_j \ne \emptyset;$ also $A_i$ and $A_j$ are comparable in size by following proposition.
\begin{proposition}

There exists $c> 0$ such that $$ c^{-1} \le \frac{r_j}{r_i} \le c$$ for all $i \in \Gamma(j), j \in I^*.$
\end{proposition}
\begin{proof}
For $i \in \Gamma(j), j \in I^*$, we consider the two cases:
\begin{itemize}
\item[Case(1).] If $i_1 \ne j_1$, then using the construction of $\mathcal{C}$ we see that $i \in \Gamma_{| O_j|}.$ Furthermore, with $r= \min\{r_k:k=1,2,\dots N\}$ we have 
\begin{equation}\label{eqn2.10}
\begin{aligned} 
r_i \le   |O_j| \le r_{i_1 i_2 \dots i_{n-1}} \le c_1 r^{-1} r_i.
\end{aligned}
 \end{equation}
Since $\epsilon < 2^{-1}c_3^{-1} \delta < \delta $, from Lemma \ref{new2} we get $ |O_j| \ge c_2^{-1} \epsilon r_j.$
Therefore we have 
\begin{equation}\label{eqn2.11}
\begin{aligned}
 c_2^{-1} \epsilon r_j \le |O_j | \le c_1 r^{-1}r_i.
 \end{aligned}
 \end{equation}
 From Equation \ref{eqn2.9}, we have $ | O_j| \le 2 c_2 \epsilon r_j +|A_j|.$ Lemma \ref{new3} and Equation \ref{eqn2.10} produce 
  \begin{equation}\label{eqn2.12}
  \begin{aligned}
  r_i \le |O_j| \le c_3(2 \epsilon +|A|)r_j.
  \end{aligned}
   \end{equation}
   Now Equation \ref{eqn2.11} and \ref{eqn2.12} yield that there exists $c > 0$ with $$ c^{-1} \le \frac{r_j}{r_i} \le c.$$
\item[Case(2).] If $i_1 = j_1$, note that $$j=j_1j_2 \dots j_kj_{k+1} \dots j_n := j_1j_2 \dots j_k j'$$
and $$i=j_1j_2 \dots j_k i_{k+1} \dots i_m := j_1j_2 \dots j_k i'$$
where $j_{k+1} \ne i_{k+1}.$ The construction of $\mathcal{B}$ gives inductively that $i' \in \Gamma(j')$ and from Case(1) we have $c^{-1} \le \frac{r_j}{r_i} \le c$. This with Lemma \ref{new1} dictates that $$(c c_1^2)^{-1} \le \frac{r_j}{r_i} \le c c_1^2. $$
With the choice $c'=  c c_1^2$ this finishes the proof.
\end{itemize}
\end{proof}
We define $ \gamma_{\epsilon} := \sup_{j \in I^*} \Gamma(j).$
\begin{theorem}
If $\gamma_{\epsilon} < \infty$ for some $\epsilon >0$ the Strong open set condition is satisfied.
\end{theorem}
\begin{proof}
Let $j^* \in I^*$ with $ \gamma_{\epsilon}= \Gamma(j^*).$ We first prove that 
\begin{equation}\label{eqn2.13}
  \begin{aligned}
  \Gamma(ij^*)=\{ij:j \in \Gamma(j^*)\}
\end{aligned}
 \end{equation}
    for any $i \in I^*.$
We can see that from the construction of the set $\mathcal{B}$ that $$ \{kj: j \in \Gamma(j^*)\} \subseteq \Gamma(kj^*) $$ for $k = 1,2, \dots , N$. On the other hand, for any $i \in I^*$ the choice of $j^*$ implies that $$\{kj: j \in \Gamma(j^*)\}= \gamma_{\epsilon}.$$
Thus the maximum of $\gamma_{\epsilon}$ forces that $ \Gamma(ij^*)= \gamma_{\epsilon}.$ From above we have $$\Gamma(ij^*)=\{ij:j \in \Gamma(j^*)\} ,~~ \forall ~ i \in I^*.$$
Now for (fixed) $1 \le k \le N$ and $i=i_1i_2\dots i_n \in I^*$ with $i_1 \ne k$, we consider the family $$ \mathcal{A}_k =\{A_l: l \in \Gamma_{| O_{ij^*}|} ~ \text{and}~ l_1 =k\}$$ where $l_1$ is the first index of $l.$ We can check that $\mathcal{A}_k$ is a cover of $A_k.$ Since $i_1 \ne l_1$, Equation \ref{eqn2.13} yields that $l \notin \Gamma(ij^*).$ Now the definition of $\mathcal{C}$ gives $ A_l \cap O_{ij^*} = \emptyset.$ At the same time from Equation \ref{eqn2.9}, one gets $ D(A_l,A_{ij^*}) \ge c_2^{-1} \epsilon r_{ij^*},$ this further produces
\begin{equation}\label{eqn2.14}
  \begin{aligned}
    D(A_k,A_{ij^*}) \ge c_2^{-1} \epsilon r_{ij^*}~~ \text{for}~~ k \ne i_1.
\end{aligned}
 \end{equation}
 Now we define a set $O_j^*=f_i(B(A, 2^{-1}c_2^{-1}\epsilon)).$ In the light of Equation \ref{eqn2.8} we have that $O_j^*$ is an open set. Furthermore we define $U:= \cup_{i \in I^*} O_{ij^*}^*.$ It is clear that $U$ is an open set and $U \cap A \ne \emptyset$. We believe that the set $U$ satisfies the condition of the SOSC. For this, let $1 \le m \le N$, we have 
 \begin{equation}
   \begin{aligned}
   f_m(U) & = \cup_{i \in I^*} f_m(O_j^*)\\ & = \cup_{i \in I^*} O_{mij^*}^*\\ & \subseteq U.
   \end{aligned}
    \end{equation}
It remains to show that $f_m(U) \cap f_n(U) = \emptyset $ for $m \ne n.$
Suppose the above does not hold for some $m \ne n. $ Then there exist $i,j$ satisfying the following condition $$ O_{mij^*}^* \cap O_{njj^*}^* \ne \emptyset .$$ Without loss of generality we can assume that $r_{mij^*} \ge r_{njj^*} .$ If $x \in O_{mij^*}^* \cap O_{njj^*}^* $ then there exist $y \in A_{mij^*}$ and $z \in A_{njj^*}$ such that $$ d(x,y) < c_2 \frac{1}{2 c_2^2}\epsilon r_{mij^*} \le \frac{c_2^{-1}\epsilon}{2} r_{mij^*}$$ and $$ d(x,z) < c_2 \frac{1}{2 c_2^2}\epsilon r_{njj^*} \le \frac{c_2^{-1}\epsilon}{2} r_{mij^*}.$$
Thanks to triangle inequality, we have $d(y,z) < c_2^{-1}\epsilon r_{mij^*}.$ Therefore, $ D(A_{mij^*},A_n) < c_2^{-1}\epsilon r_{mij^*},$ which contradicts the Equation \ref{eqn2.14}. This completes the proof.
\end{proof}
For the next note, the reader is encouraged to see \cite[Proposition $1.2$]{Lau}.
\begin{note}
Let $L_{\sigma}$ be the operator given before. We define a function $$\Phi_i(x)=\frac{(Df_i)(f_i(x)) ~~ u_{\sigma}(f_i(x))}{r(L_{\sigma}) u_{\sigma}(x)} ,$$ where $u_{\sigma}$ is the strictly positive eigenfunction of $L_{\sigma}$ corresponding to the spectral radius $r(L_{\sigma}).$ Then an operator $\tilde{L}_{\sigma}: \mathcal{C}(X) \rightarrow \mathcal{C}(X)$ defined by $$ (\tilde{L}_{\sigma}g)(x)= \sum_{i=1}^{N}\Phi_i(x) g(f_i(x))$$ satisfies the following:
\begin{itemize}
\item It can also be defined by $\tilde{L}_{\sigma}g= \frac{1}{r(L_{\sigma}) u_{\sigma}} L_{\sigma}g.$
\item $\tilde{L}_{\sigma}(1)=1.$ In other words, we have $  \sum_{i=1}^{N}\Phi_i(x) =1$ for every $x \in A.$
\end{itemize}
 
\end{note}
In view of the aforementioned note, the proof of the next lemma can be straightforwardly seen as a modification of \cite[Lemma $2.5$]{Lau}, hence omitted.
\begin{lemma}
Suppose the IFS satisfies the SOSC with an open set $U$ and $\mu$ is the generated invariant measure. Then one of the following holds.
\begin{itemize}
\item $\mu (U) \ne 0.$
\item $\mu (\partial U) \ne 0,$ where $\partial U $ denotes the boundary of $U$.
\end{itemize}
\end{lemma}

\begin{theorem} \label{mzero}
If the previous IFS satisfies the SOSC then $\mu(A_i \cap A_j)=0$ for every $i \ne j \in I^*$ with $|i| = |j|.$
\end{theorem}
\begin{proof}
Consider $U$ as the open set from the SOSC. Since $U \cup A \ne \emptyset,$ $\mu(U) \ne 0$. This together with the previous lemma yields $\mu(\partial U)=0.$ Further, by the definition of the SOSC and $ A_i \cap A_j \subseteq \overline{U}_i \cap \overline{U}_j$, we deduce $ A_i \cap A_j \subseteq \partial U_i \cap \partial U_j,$ hence the proof.
\end{proof}

\begin{definition}[\cite{Larman}]
If $Y\subset X$ and $p>0$ we denote by $N(Y,p)$ the (possibly infinite ) maximal number of disjoint closed balls with radius $p$ and centers in $B.$ If for each $0< \beta < 1$ there are constants $C$ and $D$ such that $N(U(p,x),\beta p)< C$ holds for each $D> p> 0$ and $x \in Y$ we call $Y$ an $\beta-$space.
\end{definition}
Note that \cite{Larman} each Euclidean space $\mathbb{R}^n$ is a $\beta-$space and also each compact subset of a $\beta-$space is a $\beta-$space.
\begin{definition}[\cite{Lau}]
An IFS $\{X; f_1,\dots, f_N\}$ is said to satisfy measure separated property with respect to a Borel measure $\mu$ if $\mu(A_i \cap A_j) = 0$ for every $i,j \in I^*$ with $i \ne j.$
\end{definition}
\begin{remark}\label{new5}
Note that if  $A_1, A_2, \dots, A_N$ are pairwise disjoint then the IFS will satisfy measure separated property with respect to any Borel measure. By Theorem \ref{mzero}, the IFS  $\mathcal{F}$ will satisfy the measure separated property with respect to invariant measure $\mu$ provided $\mathcal{F}$ satisfies the SOSC.
\end{remark}

\begin{lemma}\label{new4}
For an infinitesimal-similitude system with measure separated property with respect to the invariant measure $\mu$, the invariant measure $ \mu $ will satisfy $\mu(A_i)\le c_1 r_i^{\sigma_0}$ for each $i \in I^*$, where $\sigma_0= \dim_H(A).$ 
\end{lemma}
\begin{proof}
Let $k \in I^*.$ 
Since $ L_{\sigma_0}^* \mu= \mu$ and $ \big(L_{\sigma_0}^*\big)^n \mu= \mu~~ \forall n \in \mathbb{N}$, we shall prove the result for $k \in I^1 $, the result follows on similar lines to this. We have 
\begin{equation}
\begin{aligned}
\mu (A_k)&=L_{\sigma_0}^* \mu (A_k)\\ & = \sum_{i=1}^{N} ((Df_i)^{\sigma_0}\mu) \circ f_i^{-1}(A_k)\\
 & = \sum_{i=1}^{N} \int_{A_k}\Bigg(\big(Df_i\big)(f_i^{-1}(x))\bigg)^{\sigma_0} \mathrm{d}\mu ( f_i^{-1}(x))\\ & = \sum_{i=1}^{N} \int_{f_i^{-1}(A_k) \cap A}\Bigg(\big(Df_i\big)(y)\bigg)^{\sigma_0} \mathrm{d}\mu ( y)\\
 & \le  \sum_{i=1}^{N} \int_{f_i^{-1}(A_k) \cap A}R_i^{\sigma_0} \mathrm{d}\mu ( y)\\
 & \le  \sum_{i=1}^{N} c_1 r_i^{\sigma_0} \int_{f_i^{-1}(A_k) \cap A} \mathrm{d}\mu ( y)\\
 & \le  \sum_{i=1}^{N} c_1 r_i^{\sigma_0} \mu(f_i^{-1}(A_k) \cap A)\\
  &\le    c_1 r_k^{\sigma_0},
 \end{aligned}
\end{equation}
note that the fifth inequality follows from the definition of $R_i$, the sixth follows from Lemma \ref{new1}, and the last follows from the measure separated property and $\mu(A)=1.$ This completes the proof.
\end{proof}
\begin{theorem} 
If the IFS satisfies the SOSC and its attractor $A$ is a $\beta-$space then $ H^{\sigma_0}(A)>0.$
\end{theorem}
\begin{proof}
Since the involved IFS satisfies the SOSC, we have an open set $U$ such that $U\cap A \ne \emptyset.$ Let $x \in  U\cap A$, then there exists $q >0$ such that $B(x,q) \subset U.$ Let $C$ be a Borel measurable subset of $X$ with $b= \frac{|C|}{|A|} \le 1.$ Further, let $P=\{i \in \Gamma_b: A_i \cap C\ne \emptyset\}.$ By the definition of the SOSC, we have the pairwise disjoint balls $$ B(f_i(x),c_2^{-1} q r_i)\cap A \subseteq f_i(B(x,q)) \cap A.$$ Let $y \in C$, then 
\begin{equation*}
\begin{aligned}
B(y, |C|+\max_{i \in P} |A_i|+\max c_2^{-1} q r_i) & \subset B(y, b|A|+c_3 |A| \max_{i \in P} r_i+\max c_2^{-1} q r_i)\\ & 
\subset B(y, b|A|+c_3 |A| b+ c_2^{-1} q b) \\
& 
= B(y, b(|A|+c_3 |A| + c_2^{-1} q )),
\end{aligned}
\end{equation*}
from the above, it is obvious that all the aforesaid disjoint balls contained in $B(y, b(|A|+c_3 |A| + c_2^{-1} q )).$ Since $A$ is a $\beta-$space, there exists a constant $K$ such that $card(P) \le K.$ By Lemma \ref{new4} and Remark \ref{new5}, this further yields
\begin{equation*}
\begin{aligned}
\mu(C) & \le K \max_{j \in P} \mu(A_j)\\
 & \le K \max_{j \in P} c_1 r_j^{\sigma_0}\\
  & \le K  c_1 b^{\sigma_0}\\
  & \le K  c_1 |A|^{-\sigma_0} |C|^{\sigma_0},
\end{aligned}
\end{equation*}
thanks to the mass distribution principle, the result follows.
\end{proof}
\begin{remark}
The previous theorem is a generalization of \cite[ Theorem $2.4$]{Schief2} from similarity maps to infinitesimal similitudes. 
\end{remark}
\begin{theorem}
If $\gamma_{\epsilon}< \infty$ for some $\epsilon>0,$ then $A$ is a $\beta-$space.
\end{theorem}
\begin{proof}
Let us first consider $p \le |A|.$ Our aim is to show that $$ N(A,p) \le \Big(\frac{|A|}{p r_{\min}}\Big)^{\sigma_0}.$$ For this, we assume disjoint balls $C_1,C_2, \dots, C_n$ having radius $p$ and centers $x_j \in A.$ It follows that there exist $i_j \in \Gamma_{p/|A|}$ with $x_j \in A_{i_j}.$ We note that the sets $A_{i_j} \subset C_j$ are pairwise disjoint. Further, we see that the IFS $\{X;f_{i_j},j=1,2,\dots n\}$ satisfies the hypothesis of Theorem , hence the Hausdorff dimension of the attractor $A^* \subset A$ of the IFS is equal to a number $\beta \le \sigma_0$. Moreover, let $h$ be a strictly positive eigenvector corresponding to eigenvalue $r(L_{\beta}).$ That is, $L_{\beta}h = r(L_{\beta}) h.$ Since $r(L_{\beta})=1,$ we get $L_{\beta}h =  h.$ From the Equation \ref{Eqn1}, we obtain
\begin{equation*}
\begin{aligned}
h(y)& =\sum_{j =1}^n \big((Df_{i_j})(y)\big)^{\beta} h(f_{i_j}(y)) \\ & \ge \sum_{j =1}^n (r_{i_j})^{\beta} h(f_{i_j}(y)) \\ & \ge \sum_{j =1}^n (r_{i_j})^{\beta} m_h \\ & \ge n ~m_h \Big(\frac{p r_{\min}}{|A|}\Big)^{\beta}\\ & \ge n ~m_h \Big(\frac{p r_{\min}}{|A|}\Big)^{\sigma_0},
\end{aligned}
\end{equation*}
the above gives the required result.
\par
Now we consider $x \in A,$ and arbitrary positive real numbers $R$ and $ p$. For $b =\frac{R}{\epsilon}$, selecting $i \in \Gamma_b$ with $x \in A_i$, we have $B(x,R) \subset B(A_i, \epsilon r_i^*) .$ Enumerating the disjoint balls having centers in $B(A_i, \epsilon r_i^*)$, by first part of the proof we get $$ N(B(x,R),p) \le \sum_{j \in \Gamma(i)}\Big(\frac{|A_j|}{p r_{\min}}\Big)^{\sigma_0} \le  \sum_{j \in \Gamma(i)}\Big(\frac{c_3 r_j |A|}{p r_{\min}}\Big)^{\sigma_0} \le \gamma_{\epsilon} K \Big(\frac{R}{p }\Big)^{\sigma_0}, $$
for some constant $K$ which does not depend on $p$ and $R$.
\end{proof}
\begin{corollary}
If $\gamma_{\epsilon}< \infty$ for some $\epsilon>0,$ then $H^{\sigma_0}(A)>0.$
\end{corollary}
\begin{theorem}
If $A_1, A_2, \dots, A_N$ are pairwise disjoint then $ H^{\sigma_0}(A)>0.$
\end{theorem}
\begin{proof}
Define $\delta_*= \min_{i \ne j} D(A_i,A_j)$. Since $A_1, A_2, \dots, A_N$ are pairwise disjoint, $\delta_* >0.$ Without loss of generality we assume $2 c_3 \delta_* \le \delta.$
Using Equation \ref{eqn2.9} we obtain $D(A_i,A_j) \ge c_2^{-1} r_i^* \delta_*$ for incomparable pair of $i,j \in I^*.$ Let $C$ be a Borel measurable set of $X$ such that $b:=\frac{2 c_2 ~d(C)}{\delta_*} \le 1.$ From the previous line, one gets $d(C) < c_2^{-1} r_i^* \delta_*$, which further implies that there is at most one set $A_i,~~ i \in \Gamma_b,$ intersecting $C$. Hence we have 
$$ \mu(C) \le \mu(A_k) \le c_1 r_k^{\sigma_0}< c_1 b^{\sigma_0}= c_1 2^{\sigma_0} \delta_*^{- \sigma_0} d(C)^{\sigma_0}.$$ 
Now, the mass distribution principle implies $ H^{\sigma_0}(A)>0.$
\end{proof}

\section{Sub-self-infinitesimal similar sets}
Falconer \cite{Fal2} introduces the concept of sub-self similar sets. He defines a nonempty compact set $C \subset \mathbb{R}^n$ a sub-self similar set for an IFS $\{\mathbb{R}^n: f_1,f_2,\dots, f_N\}$ if $C$ satisfies $C \subseteq \cup_{i=1}^N f_i(C)$. 
\begin{proposition}
Let $\mathcal{F}=\{X; f_1,f_2, \dots, f_N\}$ be an IFS where all maps $f_i$ are contraction. Then $C$ is a sub-self-infinitesimal similar set for $\mathcal{F}$ if and only if $C= \pi(J)$ for some compact set $J \subseteq I^{\infty}$ with the property: $(i_2,i_3,\dots) \in J$ whenever $(i_1,i_2,\dots) \in J$.  
\end{proposition}
\begin{proof}
Proof follows from Proposition $2.1$ in \cite{Fal2}.
\end{proof}
We introduce the following notations. $$J=\{(i_1,i_2,\dots) \in I^{\infty}: \pi((i_n,i_{n+1},\dots) ) \in C ~~ \forall n \in \mathbb{N}\}.$$
For $n \in \mathbb{N}$, we denote $I^{\infty}_n =\{i|_n: i \in I^{\infty}\}$ and $J_n=\{i|_n: i \in J\}$ with $i|_n=i_1 i_2 \dots i_n.$
For $i \in J_n$ and $j \in J_m$, we define $ij=i_1 i_2 \dots i_n j_1 j_2 \dots j_m.$ It is easy to observe that both $i \in J_n$ and $j \in J_m$ hold whenever $ij \in J_{n+m}.$ 
For each $n \in \mathbb{N}$ and $\sigma \ge 0$, define a bounded linear operator $_nL_{\sigma}:\mathcal{C}(X) \to \mathcal{C}(X)$ by $$\big( _nL_{\sigma} g\big)(x) = \sum_{i \in J_n} (Df_i)^{\sigma} g(f_i(x)).$$

Using \cite[Theorem $5.4$]{Nussbaum4}, it can be obtained that the operator $_nL_{\sigma}$ consists a positive eigenvector corresponding to spectral radius $r(_nL_{\sigma})$ of $_nL_{\sigma}.$ Also the function $\sigma\to r(_nL_{\sigma}) $ is continuous and strictly decreasing, which produces a unique $\sigma_n$ such that $r(_nL_{\sigma_n})=1.$ From this, it is immediate to prove the following.

\begin{theorem}
Let $C$ be a sub-self-infinitesimal-similar set for $\mathcal{F}$ such that $C=\pi(J)$ for some $J \subset I^{\infty}.$ If the IFS $\mathcal{F}$ satisfies the SOSC, then $\dim_H(C)=\underline{\dim}_B(C)=\overline{\dim}_B(C)= \sigma_{\infty},$ where $\sigma_{\infty}= \lim_{n \to \infty} \sigma_n.$
\end{theorem}
\begin{proof}
Proof is left to the reader.
\end{proof}
\begin{remark}
The above theorem can be compared with \cite[Theorem $3.5$]{Fal3} and \cite[Theorem $8$]{Liu}. 
\end{remark}
\begin{remark}
Although the techniques in \cite{Fal3} and \cite{Liu} are different from that of ours, we can also follow the techniques of \cite{Fal3} and \cite{Liu} to get the result. However, we should emphasize that our technique is in more generalized form.
\end{remark}
In \cite{Fal3}, Falconer asked a question whether equality between the Hausdorff and box dimension of a sub-self-similar set holds without the OSC.
The following example conveys that equality between the Hausdorff and box dimension may not hold for sub-self-infinitesimal sets.  
\begin{example}
Let $\mathcal{F}=\{[0,1]; f_1, f_2\}$ be an IFS such that $f_1(x)=\frac{x}{1+x}$ and $f_2(x)=\frac{x+2}{3}.$ Note that $f_1$ is not a contractive similarity map but both $f_1$ and $f_2$ are infinitesimal similitude on $[0,1].$ The compact set $E=\{0,1,\frac{1}{2},\frac{1}{3}, \dots\}$ satisfies the relation $E \subseteq f_1(E) \cup f_2(E).$ Furthermore, $\dim_H(E)=0$ and $\dim_B(E)=\frac{1}{2}.$ 
\end{example} 

\section{Some corrections}
\subsection{Self-similar sets}
In \cite{SS}, Simon and Solomyak remarked that the proof of Proposition $2$ appeared in \cite{Bandt} contains an error. We notice that Schief \cite{Schief2} used the same proposition to prove Theorem $2.9$ of his paper. More precisely, he used the aforementioned proposition in part-3 of the proof. However, the theorem is correct but needs some modifications in its proof. We shall use the following lemmas to correct the proof of the aforesaid theorem. These lemmas can be treated as a generalization of \cite[Proposition $3$]{Bandt} to complete metric spaces.

\begin{lemma}\label{limdis}
Let $\{X; f_1,f_2, \dots, f_N\}$ be an IFS consists of similarity maps with similarity ratio $s_i$ and $A$ be the associated attractor. Let $\sigma_0$ be the similarity dimension of the attractor $A$, that is, $\sum_{i=1}^{N} s_i^{\sigma_0}=1$. Then we have $ \mathcal{H}^{\sigma_0}(f_i(A) \cap f_j(A) ) = 0$ for every incomparable $i, j \in I^*.$ 
\end{lemma}
\begin{proof}
It suffices to see the following: $$ \mathcal{H}^{\sigma_0}\Big(\cup_{i=1}^N f_i(A)\Big)= \mathcal{H}^{\sigma_0}(A)= \mathcal{H}^{\sigma_0}(A) \sum_{i=1}^{N} s_i^{\sigma_0}= \sum_{i=1}^{N} \mathcal{H}^{\sigma_0}(f_i(A)).$$
\end{proof}
\begin{lemma}
Let $\{X; f_1,f_2, \dots, f_N\}$ be an IFS consists of similarity maps with similarity ratio $s_i$ and $A$ be the associated attractor. Let $\sigma_0$ be the similarity dimension of the attractor $A$. Then the $\sigma_0-$dimensional Hausdorff measure $\mathcal{H}^{\sigma_0}$ satisfies the following: If $C \subseteq A,$ then $\mathcal{H}^{\sigma_0}(C)=\nu(C)$, where the outer measure $\nu$ is defined by $$\nu(C)= \inf\Big\{\sum_{i=1}^{\infty} |U_i|^{\sigma_0}: U_i~ \text{are open sets with} ~~ C \subseteq \cup_{i=1}^{\infty}U_i \Big\}.$$  
\end{lemma}
\begin{proof}
We shall prove the result for $C=A.$ Assuming it, for each $C \subseteq A$ one can establish 
\begin{equation}
\begin{aligned}
\mathcal{H}^{\sigma_0}(A)& = \mathcal{H}^{\sigma_0}(C)+\mathcal{H}^{\sigma_0}(A \backslash C)\\ & \ge 
\nu(C)+\nu(A \backslash C)\\ & \ge \nu(A) \\ & = \mathcal{H}^{\sigma_0}(A),
\end{aligned}
\end{equation}
this completes the proof of the theorem. It remains to show that $\mathcal{H}^{\sigma_0}(A)=\nu(A).$ Thanks to the definitions of $\nu$ and $\mathcal{H}^{\sigma_0}$, we get $\mathcal{H}^{\sigma_0}(A) \ge \nu(A).$ For reverse inequality it is enough to construct a covering $\{V_k\}$ such that $|V_k| < \epsilon $ and $\sum_{k} |V_j|^{\sigma_0}= \sum_{i} |U_i|^{\sigma_0}$ for a given open covering $\{U_i\}$ of $A$ and for each $\epsilon>0.$ Select $n$ with $ s_j < \epsilon \big|\cup_{i=1}^{\infty}U_i\big|$ for each $j \in I^n$, we obtain a cover $\{f_j(U_i): j \in I^n, i\}$ for $A$. This cover serves our purpose. 
\end{proof}
The notation in the next theorem are taken from \cite{Schief2}.
\begin{theorem}[\cite{Schief2}, Theorem $2.9$]
If $\mathcal{H}^{\sigma_0}(A) > 0$ then $\gamma_{\epsilon} < \infty$ for arbitrary $\epsilon.$
\end{theorem}
\begin{proof}
By the previous lemma, we note the following:
 $\mathcal{H}^{\sigma_0}(V) \le |V|^{\sigma_0}$ for any open set $V$ of $A$.
In view of Lemma \ref{limdis}, the above gives
\begin{equation*}
\begin{aligned}
\mathcal{I}(k) \mathcal{H}^{\sigma_0}(A) b^{\sigma_0} s_{\min}^{\sigma_0} & \le \sum_{j \in \mathcal{I}(k)}\mathcal{H}^{\sigma_0}(A_j)\\ & \le  \mathcal{H}^{\sigma_0}\Big(U\big(\max_{j \in \mathcal{I}(k)}|A_j|,G_k \big) \cap A\Big)\\ & \le 
(b+2b |K|)^{\sigma_0}. 
\end{aligned}
\end{equation*}
This yields $\gamma_{\epsilon} \le \frac{(1+2 |A|)^{\sigma_0}}{\mathcal{H}^{\sigma_0}(A) s_{\min}^{\sigma_0}}.$
\end{proof}
\begin{proposition}[\cite{Bandt}, Proposition $2$]
If $\dim(A)= \sigma_0$ then $\mu(A_i \cap A_j)=0$ for every $i \ne j.$
\end{proposition}
\begin{example}
Let $\big\{[0,1]; f_1(x)=\frac{x}{2},f_2(x)=\frac{x}{3}\big\}$ be an IFS. Check that the attractor $A=\{0\}$, and $\dim_H(A)=0.$ But $\mu(A_1 \cap A_2)=1$ contradicting the above proposition. 
\end{example}
\begin{example}
Let $\big\{[0,1]; f_1(x)=1,f_2(x)=\frac{x}{2}\big\}$ an IFS. Check that the attractor $A=\{0,1, \frac{1}{2},\frac{1}{4}, \dots\}$, and $\dim_H(A)=0.$ 
\end{example}
\subsection{Generalized Graph-Directed Constructions}
Equivalent to the SOSC for an IFS there is a notion of graph-directed SOSC, the interested reader can consult \cite[Definition $6.7.1$]{Edgar}. 
Here we made a modification in \cite[Theorem $4.17$]{Nussbaum1}. In particular, we have the following.
\begin{theorem}\label{GGDC}
Let us consider the GGDC with the assumptions $H4.1,~H4.2,~H4.3$ and $H4.4$ as in \cite{Nussbaum1}. 
Suppose $f_{(j,e)}|_{C_j}$ is injective for every $(j,e) \in \Gamma$ and that the GGDC satisfies the graph-directed SOSC. Then the Hausdorff dimension of each $C_i$ is equal to $r(L_{\sigma_0})$ for all $1 \le i \le p$. 
\end{theorem}
\subsection{Infinite IFS}
Mauldin and Urba\'nski \cite{MU1} devised a notion of infinite IFS.
We define Infinite IFS by $\{X; f_i, i \in I=\mathbb{N}\}$, where $X$ is a compact metric space and each $f_i: X \to X $ is a contraction map for $i \in I.$ A set $A \subset X$ is called limit set if it satisfies $$A= \cup_{i \in I} f_i(A).$$ For existence and more details, the reader can consult.
Following \cite{Nussbaum1}, let us introduce some assumptions.
\begin{itemize}
\item[A.1] Let a compact and perfect metric space $(X,d)$ and countably infinite number of contractive infinitesimal similitudes $f_i:X \to X,$ $ i\in I$. Further, suppose that there exist positive numbers $M$ and $\lambda$ with $Df_i \in K(M,\lambda)$ and $Df_i(x) > 0, ~\forall x \in X.$ Also that $\sum_{i \in I}(Df_i)(x))^{\sigma} < \infty $ for some $\sigma >0$ and $x \in X.$
\item[A.2] For each $\epsilon > 0$, there exists a $c_{\epsilon}> 1$ such that for each $i \in I$ and for all $x,y \in X$ with $0< d(x,y) < \epsilon,$ $$ c_{\epsilon}^{-1} (Df_i)(x) \le \frac{d(f_i(x),f_i(y))}{d(x,y)} \le c_{\epsilon} (Df_i)(x)$$
and $\lim_{\epsilon \to 0^+} c_{\epsilon} =1.$
\item[A.3] For each $N \ge 1$ and $1\le i < j \le N$, we assume that $f_i(A_N) \cap f_j(A_N) =\emptyset$, and for each $1 \le i \le N$, $f_i|_{A_N}$ is injective. 
\end{itemize}
Set $\sigma_{\infty}=\inf\{\sigma >0: r(L_{\sigma})< 1\},$ where $L_{\sigma}: \mathcal{C}(X) \to \mathcal{C}(X)$ is defined as 
$$ (L_{\sigma}g)(x)= \sum_{i \in I}\big((Df_i)(x)\big)^{\sigma} g(f_i(x)).$$
Now we present an improved version of \cite[Theorem $5.11$]{Nussbaum1}.
\begin{theorem}
With the assumptions $A.1,~A.2$ and for each $N \ge 1$ and $1 \le i \le N$,  $f_i|_{A_N}$ is injective, we have $\dim_H(A)=\sigma_{\infty}.$
\end{theorem}

\subsection{Continuity of Hausdorff dimension}
In \cite{AP}, Priyadarshi showed the continuity of Hausdorff dimension of the attractor of generalized graph-directed constructions under certain conditions. He also mentioned in \cite[ Remark $3.4$]{AP}, that \cite[Theorem $3.3$]{AP} can be improved further. In particular, if \cite[Theorem $2.9$]{AP} holds with the graph-directed SOSC then the condition of the graph-directed SSP can be relaxed in \cite[ Theorem $3.3$]{AP}.
\begin{theorem}
Assume that the GGDC holds all the assumptions as in \cite{AP} except the graph-directed SSC, which is replaced by the graph-directed SOSC. If $r(L_{\sigma_m,m})=1$ for all $m \ge 1$ and $r(L_{\sigma_0})=1$, then $\lim_{m \to \infty}\sigma_m= \sigma_0.$ Moreover, If the GDDC corresponding to the limit satisfies the graph-directed SOSC then the Hausdorff dimension of $C_j$ is equal to the limit $\lim_{m \to \infty}\sigma_m$ for every $1\le j \le p.$
\end{theorem}
\begin{proof}
This follows from Theorem \ref{GGDC} and \cite{AP}.
\end{proof}
\begin{example}
Let $A_n=[-\frac{1}{n},\frac{1}{n}]$ and $A=\{0\}$. Obvious that $A_n \to A$ with respect to the Hausdorff metric but the sequence $ \dim_H(A_n)$ does not converge to $\dim_H(A).$
\end{example}
\subsection{Computation of Hausdorff dimension}
In \cite{Nussbaum3,AP1,Nussbaum2}, Nussbaum and his collaborators computed Hausdorff dimension of some sets with the help of Theorem \ref{thm-nussbaum}. In particular, they estimated the lower bound for Hausdorff dimension of a set of complex continued fractions. We have mentioned in the introductory section that Theorem \ref{thm-nussbaum} covers only cantor type sets, that is, totally disconnected sets. However, Theorem \ref{mainthm} could work for a broader class of fractal sets. We strongly believe that our results will find some applications in computational part of dimension theory. 

\subsection*{Acknowledgements}
  The author is grateful to Dr. Amit Priyadarshi for many suggestions which have made this a better paper, and especially his time to elaborate his PhD thesis work.

 \bibliographystyle{amsplain}
 
\end{document}